\documentclass[11pt]{amsart}

\usepackage{amsrefs}

\usepackage{epsfig}  		
\usepackage{epic,eepic}       
\usepackage{graphicx}

\usepackage{enumerate}
\usepackage{mathbbol}
\usepackage{amssymb}
\usepackage{braket}
\usepackage{longtable}
\usepackage[colorlinks]{hyperref}
\usepackage{doi}
\usepackage{caption}
\usepackage{mathtools}
\usepackage{microtype}

\DeclareSymbolFontAlphabet{\mathbb}{AMSb}
\DeclareSymbolFontAlphabet{\mathbbl}{bbold}

\newtheorem{lemma}{Lemma}[section]
\newtheorem*{lemma*}{Lemma}
\newtheorem{theorem}[lemma]{Theorem}
\newtheorem*{theorem*}{Theorem}
\newtheorem{corollary}[lemma]{Corollary}
\newtheorem{proposition}[lemma]{Proposition}

\newtheorem*{proposition*}{Proposition}

\newtheorem*{fact*}{Fact}

\newtheorem*{notation*}{Notation}
\newtheorem*{conventions*}{Conventions}
\newtheorem{remark}[lemma]{Remark}
\newtheorem*{remark*}{Remark}

\newtheorem*{corollary*}{Corollary}

\newtheorem*{conjecture*}{Conjecture}

\newtheorem*{problem*}{Problem}
\newtheorem{question}{Question}
\newtheorem*{question*}{Question}

\newtheorem{assumption*}{Assumption}

\theoremstyle{definition}

\newtheorem*{example*}{Example}
\newtheorem{definition}[lemma]{Definition}
\newtheorem*{definition*}{Definition}

\theoremstyle{remark}

\newtheorem*{claim*}{Claim}

\newtheorem*{case*}{Case}
\newtheorem*{construction*}{Construction}

\newtheorem*{exercise*}{Exercise}

\numberwithin{equation}{section}

\newcommand{\bs}{\backslash}

\newcommand\BB{{\mathcal B}}
\newcommand\CC{{\mathcal C}}
\newcommand\DD{{\mathcal D}}

\newcommand\FF{{\mathcal F}}
\newcommand\GG{{\mathcal G}}

\newcommand\LL{{\mathcal L}}

\newcommand\PP{{\mathcal P}}

\newcommand\<{\langle}
\renewcommand\>{\rangle}

\newcommand\cont{2^{\aleph_0}}

\newcommand{\ra}{\rightarrow}
\newcommand{\Ra}{\Rightarrow}

\newcommand{\La}{\Leftarrow}

\def\Ind#1#2{#1\setbox0=\hbox{$#1x$}\kern\wd0\hbox to 0pt{\hss$#1\mid$\hss}
\lower.9\ht0\hbox to 0pt{\hss$#1\smile$\hss}\kern\wd0}

\def\notind#1#2{#1\setbox0=\hbox{$#1x$}\kern\wd0
\hbox to 0pt{\mathchardef\nn=12854\hss$#1\nn$\kern1.4\wd0\hss}
\hbox to 0pt{\hss$#1\mid$\hss}\lower.9\ht0 \hbox to 0pt{\hss$#1\smile$\hss}\kern\wd0}

\def\includeE#1{{\lhook\kern-3.5pt\joinrel\smash{
		\mathop{\longrightarrow}\limits^{#1}}}}

\def\efor/{Example~\ref{E4}}

\def\BL/{Baldwin--Lachlan}
\def\Bu/{Buechler}
\def\Hr/{Hrushovski}
\def\lm/{locally modular}
\def\wm/{weakly minimal}
\def\nm/{non--modular}
\def\ss/{superstable}
\def\ud/{unidimensional}
\def\sm/{strongly minimal}

\def\hbar{\bar{h}}

\def\tr/{trivial}
\def\nt/{non--trivial}
\def\st/{strong type}

\def\Fa0{{\FF^a_{\aleph_0}}}

\def\<{\langle}
\def\>{\rangle}

\newcommand{\Geom}{\mathrm{Geom}}
\newcommand{\Mon}{\mathrm{Mon}}
\newcommand{\Av}{\mathrm{Av}}
\renewcommand{\phi}{\varphi}
\newcommand{\lex}{<_{lex}}

\newcommand{\clex}{\prec_{lex}}
\newcommand{\clexeq}{\preceq_{lex}}

\newcommand{\grant}{This paper is part of a project that has received funding from the 
	European Research Council (ERC) under the European Union's Horizon 2020 
	research and innovation programme (grant agreement No 810115 - Dynasnet). The author is further supported by Project 21-10775S of the Czech Science Foundation (GA\v{C}R)}

\title{Decidability in geometric grid classes of permutations}

\author{Samuel Braunfeld}

\thanks{\noindent \grant}


\begin{document}
	\begin{abstract}
		We prove that the basis and the generating function of a geometric grid class of permutations $\Geom(M)$ are computable from the matrix $M$, as well as some variations on this result. Our main tool is monadic second-order logic on permutations and words.
	\end{abstract}

\maketitle

\section{Introduction}

The study of permutation classes is focused on hereditary families of permutations, which are closed downwards under a natural notion of containment analogous to induced subgraph. Such a class is determined by its {\em basis}, i.e. the set of minimal forbidden permutations. A recurring problem, going back to the origin of the field in Knuth's result that the stack-sortable permutations are precisely those forbidding $(231)$ \cite{knuth1968art}, is that one is given an abstract description of a permutation class and wants to determine an explicit description by giving the basis. We will be concerned with such a question for {\em geometric grid classes}, as well as the problem of determining the generating function. 

As the name suggests, geometric grid classes are permutation classes defined by certain simple geometric constraints. Namely, one fixes a finite grid in the plane, for each cell chooses to leave it empty or to draw either a diagonal line or an anti-diagonal line, and considers the class of all permutations that can be plotted on the resulting set of diagonal and anti-diagonal lines (see Figure \ref{fig:geom}). They have found diverse applications, including the enumeration of specific permutation classes \cite{pantone2017enumeration}, the characterization of ``small'' permutation classes with growth rate below a certain threshold \cite{albert2015inflations}, the algorithmic analysis of permutation classes of polynomial growth \cite{homberger2016effective}, strengthenings of well-quasi-order \cite{brignall2022labeled}, and connections with the graph parameter lettericity \cite{alecu2022letter}. These classes were introduced in \cite{albert2013geometric}, which proved that they have finite bases and rational generating functions. These proofs relied on an explicit encoding of permutations in a given geometric grid class into the words of a regular language, but were ultimately non-constructive, fundamentally relying on Higman's lemma that the set of finite words in a finite alphabet contains no infinite antichain in the subword ordering. In particular, the style of argument does not determine the size of the basis or an upper bound on the size of the basis elements, and so the naive attempt to compute the basis by going through all permutations in increasing size and checking whether each is a basis element will eventually find all the basis elements, but cannot determine when it has actually finished and should halt. 

This non-constructiveness represented a gap in the understanding of these relatively simple classes, and has presented an obstacle to their algorithmic analysis (e.g., see the comments in \cite{bassino2017algorithm}). Thus the conclusion to \cite{albert2013geometric} posed the problem of proving it is computable, given the grid specifying a geometric grid class, to determine its basis and generating function (see also \cite[\S 5.2]{huczynska2015well}). We give an upper bound on the size of basis elements and a resolve these problems.

\begin{theorem} [Proposition \ref{prop:bound}, Theorems \ref{thm:compBasis}, \ref{thm:compgf}]
	Let $M$ represent a finite $0/$$\pm1$-matrix and $\Geom(M)$ the corresponding geometric grid class. Then the size of the basis elements of $\Geom(M)$ are bounded above by a computable (in fact, elementary) function of $|M|$, and the basis of $\Geom(M)$ and the generating function of $\Geom(M)$ are computable from $M$.

\end{theorem}

Our methods also prove corresponding results for certain subclasses of geometric grid classes (Theorem \ref{thm:strongComp}), some of which were also asked about in \cite{albert2013geometric}. A little further analysis yields the computability of the basis for the substitution closure of a geometric grid class, studied in \cite{albert2015inflations}.

\begin{theorem} [Theorem \ref{thm:compSC}]
	Let $M$ represent a finite $0/$$\pm1$-matrix and $\Geom(M)$ the corresponding geometric grid class. Then the basis of the substitution closure of $\Geom(M)$ is computable from $M$.
\end{theorem}

Our main tool is monadic second-order logic on permutations and on words. While monadic second-order logic on words is, in a formal sense, equivalent to working with regular languages, we find it much more tractable from a practical standpoint. While computing the basis of $\Geom(M)$ ultimately shows that $\Geom(M)$ is definable from a sentence of first-order logic that is computable from $M$, the key initial step is that $\Geom(M)$ is definable by a sentence of monadic second-order logic computable from $M$.

The particularly simple encoding of elements of $\Geom(M)$ as words is the key to their combinatorial properties,  and Proposition \ref{prop:qfinterp} shows that the existence of such an encoding actually characterizes subclasses of geometric grid classes, and similarly characterizes graph classes of bounded lettericity.

\section{Preliminaries}

Our preliminaries will be minimalist and sometimes informal. For a fuller treatment of permutations and geometric grid classes, we refer to \cite{albert2013geometric}, particularly the second section. For a fuller treatment of monadic second order logic, we refer to \cite{thomas1997languages}, particularly the second section, or \cite{torun}.

\subsection{Permutations and geometric grid classes} \label{sec:prePerm}

We define a {\em permutation} to be a set equipped with two linear orders $<_1$ and $<_2$. At least in the finite, this can be viewed as one order giving the ``standard'' order and the other giving the ``permuted'' order. Embeddings between permutations then agree with the standard notion of {\em containment} between permutations. Unless stated otherwise, permutations will be assumed to be finite.

A set of points in the plane such that no two points have the same $x$ or $y$-coordinate naturally gives rise to a permutation, with $<_1$ defined as the left-to-right order and $<_2$ as the bottom-to-top order. Since the permutation so defined is not sensitive to the exact coordinates of the points but only to their relative orderings, many such sets of points will represent the same permutation.

All matrices in this paper will be assumed to be finite and have entries in $\set{0, \pm 1}$. Given an $m \times n$ matrix $M$, we now define the {\em geometric grid class} $\Geom(M)$ associated to $M$. First, one considers an $m \times n$ grid of unit squares in the plane. In each cell, if the corresponding entry of $M$ is 1 then one draws an open line segment from the bottom left to top right corner of the cell, if the corresponding entry is $-1$  then one draws an open line segment from the top left to bottom right corner, and if the corresponding entry is 0 then the cell is left empty. These line segments are the {\em standard figure} of $M$. Then $\Geom(M)$ consists of all permutations that can be represented by points drawn on the standard figure. See Figure \ref{fig:geom} for an example.

\begin{figure}
	\includegraphics[scale=.65]{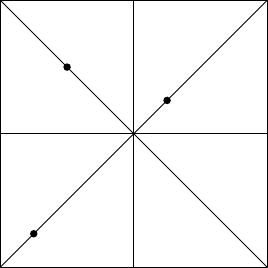}
	\caption{The standard figure for the matrix $\begin{pmatrix} -1 & 1 \\ 1 & -1 \end{pmatrix}$, with one possible gridding of the permutation $(132)$.}
	\label{fig:geom}
\end{figure}

We remark that, given $M$, it is not, in general, equivalent to consider the permutations that can be drawn in the cells of the grid corresponding to non-zero entries of $M$ so that the points in each cell are increasing if the entry of $M$ is 1 and decreasing if the entry of $M$ is $-1$. This is called the {\em monotone grid class} associated to $M$. See \cite{albert2013geometric}*{Figure 4} for an example where this differs from $\Geom(M)$.

Given a set $\BB$ of permutations, we let $\Av(\BB)$ be the set of permutations that do not embed any element of $\BB$. Since $\Geom(M)$ is closed under embeddings, there is some set $\BB$ such that $\Geom(M) = \Av(\BB)$. The set $\BB$ is called the {\em basis} of $\Geom(M)$, and consists of the embedding-minimal permutations not contained in $\Geom(M)$.

\begin{theorem}[\cite{albert2013geometric}*{Theorem 6.2}]
Every geometric grid class has a finite basis.
\end{theorem}

We now continue to some technical points about geometric grid classes that we will use.

A matrix $M$ is a {\em partial multiplication matrix} if we can assign every column and every row either 1 or $-1$ (called the {\em column signs} and {\em row signs}) such that every non-zero entry of $M$ equals the product of its row and column signs. By \cite{albert2013geometric}*{Proposition 4.2} every geometric grid class $\Geom(M)$ is the geometric grid class of a partial multiplication matrix computable from $M$; so we lose no generality in assuming our geometric grid classes are always presented via a partial multiplication matrix. Given a partial multiplication matrix, to every non-zero cell in the corresponding grid we can assign one corner (which will be one of the endpoints of cell's line segment in the standard figure) to be its {\em origin} as follows: if the row sign is 1 (resp. $-1$), then the origin is on the bottom (resp. top), and if the column sign is 1 (resp. $-1$) then the origin is on the left (resp. right).

Associated to a matrix $M$, we now also define the class $\Geom^\sharp(M)$. This consists of permutations additionally equipped with unary relations $C_1, \dots, C_n$ corresponding to the non-zero cells of $M$ such that the underlying permutation is in $\Geom(M)$, the unary relations partition the domain, and the underlying permutation can be drawn on the standard figure of $M$ in a manner agreeing with the unary relations. We will refer to elements of $\Geom^\sharp(M)$ as {\em $M$-gridded permutations}.

Given a partial multiplication matrix $M$ and a choice of row and column signs, let $\Sigma = \set{C_1, \dots, C_n}$ be its set of non-zero cells. Let $\Sigma^*$ be the set of finite words in the alphabet $\Sigma$, with embeddings given by the subword (rather than factor) relation. We now define a map $\phi_M^\sharp \colon \Sigma^* \to \Geom^\sharp(M)$ as follows: given $w \in \Sigma^*$, the $i^{th}$ letter of $w$ is mapped to a point on the standard figure of $M$, into the cell agreeing with the letter; furthermore, the distance each point is placed from the origin of its cell is strictly increasing with respect to the index in $w$. It is straightforward to check this entirely determines $\phi^\sharp$. Similarly, we define $\phi_M \colon \Sigma^* \to \Geom(M)$ by taking the underlying permutation of the $\phi^\sharp$-image. See \cite{albert2013geometric}*{Figure 5} for an example.

\begin{proposition}[\cite{albert2013geometric}, Proposition 5.1] \label{prop:phiFacts}
	The mappings $\phi_M$ and $\phi_M^\sharp$ are length-preserving, finite-to-one, onto, and order-preserving with respect to embeddings.
\end{proposition}

\subsection{Monadic second order logic} \label{sec:preMSO}
We begin by fixing a language $\LL$, consisting of relation symbols with fixed arities, and always assumed to contain equality. For example, the language of permutations is $\set{<_1, <_2}$, where both relations are binary. We then consider well-formed expressions that can be built using variables, the relations of the language, boolean connectives, and quantifiers. In monadic second-order logic (MSO), variables are  allowed to represent both elements and subsets (i.e. unary relations) of the underlying structure. Following standard notation, lowercase variables will correspond to elements while uppercase variables will correspond to sets. Such a well-formed expression is called a {\em formula}, and a {\em sentence} is a formula where every variable is in the scope of some quantifier.

For example, a permutation is {\em sum-indecomposable} if it cannot be partitioned into two sets such that one set is entirely above the other in both orders. This can be expressed by a sentence Sum-Ind defined as follows.
\begin{align*}
\neg \exists X, Y (&\forall x ((x \in X \vee x \in Y) \wedge \neg (x \in X \wedge x \in Y)) \wedge \\
&\forall x, y ((x \in X \wedge y \in Y) \ra (x <_1 y \wedge x <_2 y)))
\end{align*}

As seen in the above example, we will often abuse notation and, when $X$ is a unary relation, write $c \in X$ for $X(c)$.

Similarly, a permutation is {\em skew-indecomposable} if it cannot be partitioned into two sets $X, Y$ such that $X <_1 Y$ and $X >_2 Y$, and this can be expressed by a similar sentence.

Given a sentence $\theta$ and a structure $M$ (in the same relational language), we write $M \models \theta$ (read ``$M$ satisfies $\theta$'') if $\theta$ is true when interpreted in $M$. For example, if $\pi$ is a permutation, then $\pi \models$ Sum-Ind if and only if $\pi$ is sum-indecomposable. That is, Sum-Ind {\em defines} the set of sum-indecomposable permutations. (A formal treatment of formulas and of the satisfaction relation $\models$ requires involved inductions, and should be included in any standard treatment of MSO, e.g. \cite{torun}*{\S 2}.)

A {\em theory} is a set of sentences (in a fixed language). Given a set of structures $\CC$, the {\em theory of $\CC$} is $Th(\CC) := \set{\theta | \forall M \in \CC, M \models \theta}$.

We now turn to words in MSO. Given an alphabet $\Sigma$, the language of $\Sigma$-words is $\set{<, \set{U_\sigma}_{\sigma \in \Sigma}}$, where $<$ is a binary relation and each $U_\sigma$ is unary. A word is then a set equipped with these relations, where $<$ is a linear order and the unary relations partition the domain. Embeddings between words viewed as such structures corresponds to the subword relation. The set of all $\Sigma$-words will be denoted $\Sigma^*$.

The following theorem underpins the decidability results of this paper.

\begin{theorem}[\cite{torun}*{Theorem A.4}] \label{thm:MSOreg}
	Fix a finite alphabet $\Sigma$. Then $L \subset \Sigma^*$ is a regular language if and only if there is a sentence $\theta$ (in the language of $\Sigma$-words) such that $L = \set{w \in \Sigma^* | w \models \theta}$. Furthermore, the translations in both directions between a finite automaton accepting $L$ and such a sentence $\theta$ are computable.
\end{theorem}

Since we can computably check whether a finite automaton accepts all strings, we obtain the following.

\begin{corollary} [\cite{torun}*{Theorem A.5}] \label{cor:s1s}
	Let $\Sigma$ be a finite alphabet. Then given any sentence $\theta$, it is computable whether $\theta \in Th(\Sigma^*)$.
\end{corollary}

When the conclusion of Corollary \ref{cor:s1s} holds with a class $\CC$ in place of $\Sigma^*$, we will say the theory of $\CC$ is {\em decidable}. We now introduce a method to reduce decidability results from one class to another. Although the MSO theory of all permutations is undecidable (e.g., see the comments about ``two-dimensional words'' in \cite{thomas1997languages}),  this will yield the decidability of the theory of each geometric grid class. Let $\CC$ and $\DD$ be two classes of finite structures, possibly in different relational languages $\LL_C$ and $\LL_D$. An {\em interpretation of $\CC$ in $\DD$} is a surjective map $I \colon \DD \to \CC$ such that there is a uniform definition of $I(D)$ from $D$ for every $D \in \DD$, i.e. for every relation $S(x_1, \dots, x_n)$ in $\LL_\CC$, there is an $\LL_\DD$-formula $R_S(x_1, \dots, x_n)$ such that for every $D \in \DD$, we have that $I(D)$ is isomorphic to the $\LL_C$-structure obtained by taking the domain of $D$ and interpreting $S$ as the relation $\set{(d_1, \dots, d_n) \in D^n | D \models R_S(d_1, \dots, d_n)}$. (For simplicity, our definition of interpretation is much more restrictive than usual. In our case, each $R_S$ will be quantifier-free and we are working with hereditary classes, so the usual extra freedom to define a subdomain of $D$ is not needed.)

For a simple example, recall that the {\em inversion graph} of a permutation $\pi$ is the graph with the same domain that places an edge between $x$ and $y$ if and only if $<_1$ and $<_2$ disagree between $x$ and $y$. If $\DD$ is a class of permutations and $\CC$ the corresponding class of inversion graphs, then $\CC$ is interpretable in $\DD$ and the formula $R_E(x, y)$ defining the edge relation $E$ is $R_E(x,y) := (x <_1 y \wedge x >_2 y) \vee (x <_2 y \wedge x >_1 y)$.

\begin{theorem}[\cite{torun}*{Theorem B.1}] \label{thm:interpRed}
	Let $\CC$ and $\DD$ be two classes of finite structures in languages $\LL_\CC$ and $\LL_\DD$ with $I \colon \DD \to \CC$ an interpretation. Then for every $\LL_\CC$-sentence $\theta$, there is an $\LL_\DD$-sentence $\theta'$ computable from $\theta$ and $I$ such that $\theta \in Th(\CC) \iff \theta' \in Th(\DD)$.
	
	 In particular, if $Th(\DD)$ is decidable, then so is $Th(\CC)$.
\end{theorem}

Given $\theta$, the sentence $\theta'$ in the theorem above is obtained by replacing all occurrences of the relations of $\LL_\CC$ by the $\LL_\DD$-formulas defining them in the interpretation $I$.

\subsection{Why MSO? \nopunct} \label{sec:why} In light of Theorem \ref{thm:MSOreg} showing the equivalence between regular languages, which were the basic tool of \cite{albert2013geometric}, and MSO on words, one may wonder what is gained by using MSO instead. First, MSO may be seen as a high-level language that Theorem \ref{thm:MSOreg} ``compiles'' into an automaton, making it easier to describe complicated regular languages. A formal manifestation of this is the result that in general, the size of the finite automaton corresponding to an MSO sentence might not even be elementary (i.e. bounded above by a tower function) in the length of the sentence \cite{reinhardt2002complexity}*{Theorem 13.1}, although we will see the blow-up is more controlled in our setting.

A second reason is that working with MSO allows us to reason directly with permutations in many cases, while the work of translating back to words is abstracted away into an interpretation. And since the permutations in a geometric grid class have natural distinguished subsets corresponding to the cells of various griddings, MSO is particularly well-suited to them.

\subsection{Notation} For the rest of the paper, $M$ will represent a (not necessarily fixed) partial multiplication matrix with a computed assignment of row/column signs. We let  $\Sigma = \set{C_1, \dots, C_n}$ be the non-zero cells of $M$, with $m_i = \pm 1$ as the entry of $C_i$ (so $\Sigma$, $n$, and each $m_i$ all silently depend on $M$). We let $\LL_P = \set{<_1, <_2}$ be the language of permutations, $\LL_\Sigma = \set{< , C_1, \dots, C_n}$ be the language of $\Sigma$-words, and $\LL_{P,M} = \set{<_1, <_2, C_1, \dots, C_n}$ where each $<_i$ is binary and each $C_i$ is unary be the language of $M$-gridded permutations.

\section{Computing the basis of Geom(M)} \label{sec:basis}
In this section, we compute the basis of $\Geom(M)$ and afterwards show we can bound the length of the basis elements. A similar approach to ours was developed in \cites{lagergren, cattell2000computing} to compute the minimal forbidden minors of certain minor-closed graph classes of bounded tree-width. Our setting is simpler, since we reduce to the classical theory of regular languages of words, rather than using Courcelle's theory of graph grammars.

\subsection{An MSO definition of Geom(M)}

In this subsection, we will describe an $\LL_P$-sentence $\Geom_M$ defining $\Geom(M)$.

We first write an $\LL_P$-formula $\Mon_M(X_1, \dots, X_n)$ such that the sentence  $\exists Y_1, \dots, Y_n \Mon_M(Y_1, \dots, Y_n)$ defines the monotone grid class of $M$. Let $\Mon_M(X_1, \dots, X_n)$ be the conjunction of the following clauses.
\begin{enumerate}
	\item (The cells form a partition) $\forall x( \bigvee_{i \in [n]} x \in X_i \wedge \bigwedge_{i\neq j \in [n]} (x \in X_i \ra x \not\in X_j))$
	\item (Each cell is increasing or decreasing as specified) For each $i \in [n]$ with $m_i = 1$ (resp. $m_i= -1)$: $\forall x,y (x, y \in X_i \wedge x < _1 y) \ra x <_2 y$ (resp. $\forall x,y (x, y \in X_i \wedge x < _1 y) \ra x >_2 y$)
	\item (Relations between cells) For every $i, j \in [n]$ with $C_i$ left of (resp. below) $C_j$: $\forall x,y (x \in X_i \wedge y \in X_j) \ra x <_1 y$ (resp. $\forall x,y (x \in X_i \wedge y \in X_j) \ra x <_2 y$)
\end{enumerate}

We next define a relation $D(X_1, \dots X_n; x, y)$. This relation will eventually appear with the set-variables quantified over, and so should be thought of as a binary relation on singletons; the idea is that it will hold of $x, y$ if they are in the same row or column as specified by the set-variables and $x$ is closer to its cell's origin than $y$ to its cell's origin. Let $D(X_1, \dots X_n; x, y)$ be the disjunction of the following clauses.
\begin{enumerate}
	\item For each $i,j \in [n]$ (allowing $i = j$) such that $C_i, C_j$ are in the same row and the row sign is 1 (resp. $-1$): $x \in X_i \wedge y \in X_j \wedge x <_2 y$ (resp. $x \in X_i \wedge y \in X_j \wedge x >_2 y$)
	\item For each $i,j \in [n]$ (allowing $i = j$) such that $C_i, C_j$ are in the same column and the column sign is 1 (resp. $-1$): $x \in X_i \wedge y \in X_j \wedge x <_1 y$ (resp. $x \in X_i \wedge y \in X_j \wedge x >_1 y$)
\end{enumerate}

Given a binary relation $R(x, y)$, we write a sentence $\mathrm{Acyc}_R$ specifying (when interpreted in a finite structure) that $R$ does not have a directed cycle. Let $\mathrm{Acyc}_R := \neg \exists X \forall x \in X (\exists y \in X R(y, x) \wedge \exists z \in X R(x, z))$.

Finally, we define the $\LL_P$-sentence $\Geom_M$ as follows.
\[ \Geom_M := \exists Y_1, \dots, Y_n (\Mon_M(Y_1, \dots, Y_n) \wedge \mathrm{Acyc}_{D(Y_1, \dots, Y_n; x, y)}) \]

We remark that Clause (2) of $\Mon_M(Y_1, \dots, Y_n)$ is in a sense redundant in $\Geom_M$, since it is implied by $\mathrm{Acyc}_{D(Y_1, \dots, Y_n; x, y)}$. That is, if there were two points in $Y_i$ that went against the direction of $C_i$, then they would form a directed $D$-cycle of size 2. Thus Clause (2) of $\Mon_M(Y_1, \dots, Y_n)$ will not be invoked in the proof of the following proposition.

\begin{proposition} \label{prop:msoDef}
Let $\pi$ be a finite permutation. Then $\pi \in \Geom(M) \iff \pi \models \Geom_M$.
\end{proposition}
\begin{proof}
	$(\Ra)$ Let $\pi \in \Geom(M)$, and consider a particular gridding $\pi^\sharp$. Then if we interpret $Y_i$ as the set of points of $\pi^\sharp$ in $C_i$, this will witness the existentials in $\Geom_M$.
	
	$(\La)$ Suppose $\pi \models \Geom_M$. Then there exist $Y_1, \dots, Y_n \subset \pi$ as specified. In particular, $D(Y_1, \dots, Y_n; x,y)$ defines a directed acyclic graph, and so it can be completed to a linear order $\prec$. Consider the word $w_\pi$ of length $|\pi|$, whose $i^{th}$ element is $C_i$ if the $i^{th}$ element of $\pi$ in the $\prec$-order is in $Y_i$ (this is well-defined by Clause (1) of $\Mon_M(Y_1, \dots, Y_n)$). Then $\pi' = \phi_M^\sharp(w_\pi)$ is an $M$-gridded permutation, and we claim it is isomorphic (as a permutation) to $\pi$.
	
	Given $x \in \pi$, let $x'$ be the the corresponding element of $\pi'$, i.e. if $x$ is the $i^{th}$ element in the $\prec$-order of $\pi$, then $x'$ is the image of the $i^{th}$ element of $w_\pi$. We will show $x \mapsto x'$ is an isomorphism. Let $x, y \in \pi$. Suppose $x <_1 y$. Let $Y_i, Y_j$ be such that $x \in Y_i, y \in Y_j$. Then $x' \in C_i, y' \in C_j$. If $C_i$ and $C_j$ are in different columns, then Clause (3) in $\Mon_M(Y_1, \dots, Y_n)$ forces that $C_i$ is left of $C_j$. So $x' <_1 y'$.
	
	So suppose that $C_i$ and $C_j$ are in the same column. We further assume the column sign is 1, since the case with column sign $-1$ is symmetric. Since $x <_1 y$, Clause (2) in $D(Y_1, \dots, Y_n ; x,y)$ implies that $\pi \models D(Y_1, \dots, Y_n; x,y)$, and so $x \prec y$. Thus $x$ appears before $y$ in $w_\pi$, so $x'$ is closer to the origin of $C_i$ than $y'$ is to the origin of $C_j$. Since the sign of their common column is 1, this implies $x' <_1 y'$.

	Swapping $x$ and $y$ shows that $y <_1 x \Ra y' <_1 x'$, so $x <_1 y \iff x' <_1 y'$. The argument that $x <_2 y \iff x' <_2 y'$ is essentially identical.
\end{proof}

\begin{remark}
	{\em In MSO, one can state that a binary relation is acyclic via a sentence that works even in infinite structures. Using this in place of our definition of $\mathrm{Acyc}$, the hypothesis that $\pi$ is finite could be dropped, if we define $\Geom(M)$ by its basis when considering permutations larger than $\cont$.}
	\end{remark}

\subsection{Computing a basis}

We wish to reduce the problem of finding a basis for $\Geom(M)$ to a problem on finite words, so we begin by showing $\Geom(M)$ can be interpreted in a set of finite words. However, since the basis elements lie outside $\Geom(M)$, we must consider a larger class containing $\Geom(M)$ and its basis elements. As mentioned in Section \ref{sec:preMSO}, the class of all permutations does not have a decidable theory, and so is too large. But we will see that we may extend $\Geom(M)$ to a larger geometric grid class containing all its basis elements, thus staying within the realm of decidability.

Recall the maps $\phi_M$ and $\phi^\sharp_M$ from the end of Section \ref{sec:prePerm}.

\begin{lemma} \label{lem:interp}
	The map $\phi_M$ is an interpretation of $\Geom(M)$ in $\Sigma^*$, and $\phi^\sharp_M$ is an interpretation of $\Geom^\sharp(M)$ in $\Sigma^*$. 
\end{lemma}
\begin{proof}
	By Proposition \ref{prop:phiFacts}, $\phi_M$ and $\phi^\sharp_M$ are surjective, so it remains to define the relations $R_{<_1}$, $R_{<_2}$, and $R_{C_i}$. We define $R_{<_1}(x,y)$ as the disjunction of the following clauses, and $R_{<_2}(x,y)$ is defined similarly.
	\begin{enumerate}
		\item For all $i, j \in [n]$ such that $C_i$ is left of $C_j$: $x \in C_i \wedge y \in C_j$
		\item For all $i, j \in [n]$ such that $C_i$ and $C_j$ are in the same column, and the column sign is 1 (resp. $-1)$: $x \in C_i \wedge y \in C_j \wedge x < y$ (resp. $x \in C_i \wedge y \in C_j \wedge x > y$)
	\end{enumerate}
Finally, we define $R_{C_i}(x)$ as $x \in C_i$. 
	\end{proof}

\begin{corollary}
 It is computable from $M$ whether any given $\LL_P$-sentence $\theta$ is in $Th(\Geom(M))$.
\end{corollary}
\begin{proof}
	By Corollary \ref{cor:s1s} this is true for $Th(\Sigma^*)$. The result follows from the fact that the relations $R_{<_1}$ and $R_{<_2}$ in the interpretation in Lemma \ref{lem:interp} are computable from $M$ and from Theorem \ref{thm:interpRed}.
	\end{proof}

\begin{definition}
	Given a class $\CC$ of structures, we let $\CC^{+1}$ be the class of {\em one-point extensions} of structures in $\CC$, i.e. a structure $\pi \in \CC^{+1}$ if there is some $x \in \pi$ such that $(\pi \bs \set{x}) \in \CC$.
\end{definition}

\begin{theorem}[\cite{albert2013geometric}*{Theorem 6.4}] \label{thm:onepoint}
	Let $M$ be a matrix with dimensions $t \times u$. Then for any matrix $N$ containing all matrices with dimensions $(3t+1) \times (3u+1)$, we have $\Geom(M)^{+1} \subset \Geom(N)$.
	\end{theorem}

\begin{theorem} \label{thm:compBasis}
	The basis of $\Geom(M)$ is computable from $M$.
\end{theorem}
\begin{proof}
	Given $M$, by Theorem \ref{thm:onepoint} we can compute a matrix $N$ such that $\Geom(M)^{+1} \subset \Geom(N)$. Thus $\Geom(N)$ contains the basis of $\Geom(M)$. Given $\BB \subset \Geom(N)$ finite, we now define a sentence $\mathrm{Basis}_{\BB, M}$ that is in $Th(\Geom(N))$ if and only if $\BB$ contains the basis of $\Geom(M)$. Let $\mathrm{Basis}_{\BB, M} :=$ 
	 \[\neg \Geom_M \ra \bigvee_{B \in \BB}(\text{there exist elements inducing a copy of $B$}) \]
	
	We can now compute the basis of $\Geom(M)$. Note that we can easily check whether a given permutation is a basis element for $\Geom(M)$. So start with $\BB = \emptyset$ and iterate through all the permutations in increasing length, adding each new basis element to $\BB$. Each time we add a new element to $\BB$, we check whether $\mathrm{Basis}_{\BB, M} \in Th(\Geom(N))$. When this is true, $\BB$ is the basis of $\Geom(M)$. 
\end{proof}

\subsection{Bounding the size of basis elements}

The algorithm in Theorem \ref{thm:compBasis} gives no upper bound on the size of the basis elements of $\Geom(M)$. Here we give such a bound, although it is almost certainly very far from optimal. This gives an alternative algorithm for computing the basis that simply checks all permutations up to the maximum size, and also bounds the run-time of the previous algorithm.

\begin{proposition} \label{prop:bound}
	There is a computable function giving an upper bound on the size of the largest basis element of $\Geom(M)$ in terms of $|M|$.
\end{proposition}
\begin{proof}
	Given $M$, we can write an $\LL_P$-sentence $\mathrm{Basis}_P$ defining the basis $\BB_M$ of $\Geom(M)$, saying $\neg \Geom_M$, and for all singletons $x$ there is subset $X$ such that every $y \neq x$ is in $X$ and the restriction to $X$ satisfies $\Geom(M)$. The length of this sentence is bounded in terms of $|M|$. By Theorem \ref{thm:onepoint}, we may compute a larger matrix $N$ so $\Geom(N)$ contains $\Geom(M)^{+1}$. Using Theorem \ref{thm:interpRed} and Lemma \ref{lem:interp}, we translate the $\LL_P$-sentence $\mathrm{Basis}_P$ to a sentence $\mathrm{Basis}_{\Sigma_N}$ in the language of $\Sigma_N$-words (where $\Sigma_N$ is the set of non-zero cells of $N$) defining $\phi_N^{-1}(\BB_M)$ and with length bounded in terms of $|N|$. By Theorem \ref{thm:MSOreg}, we compute a finite automaton that accepts $\phi_N^{-1}(\BB_M)$, with size bounded by a computable function of the length of $\mathrm{Basis}_{\Sigma_N}$. But this automaton accepts a finite language, so by the pumping lemma, the length of a word it accepts cannot be larger than the number of states of the automaton. So its size bounds the length of the elements in $\phi_N^{-1}(\BB_M)$, and thus in $\BB_M$ since $\phi_N$ is length-preserving.
	\end{proof}

As noted in Section \ref{sec:why}, in general the size of the automaton produced by Theorem \ref{thm:MSOreg} need not be elementary in the length of the corresponding sentence, so a priori neither is the function in Proposition \ref{prop:bound}. However, by tracing through the construction, one sees (e.g. \cite[Remark 12.22]{weyer2002decidability}) that the size of the automaton corresponding to a sentence $\theta$ in prenex normal form with $qr_\theta$-many quantifier alternations is bounded above by the iterated exponential function $\mathrm{exp}_{qr_\theta+2}(|\theta|)$, where $\mathrm{exp}_0(n) = n$ and $\mathrm{exp}_{k+1}(n) = 2^{\mathrm{exp}_k(n)}$. Since the quantifier-rank of $\Geom_M$ in prenex normal form is bounded and fairly small (and so the same is true of the sentence $\mathrm{Basis}_P$ in Proposition \ref{prop:bound}), and the interpretation in Lemma \ref{lem:interp} uses quantifier-free formulas with length polynomial in $|N|$ and so exponential in $|M|$, we obtain an iterated exponential function of similarly small height as an upper bound for the function in Proposition \ref{prop:bound}.

\section{Computing the generating function of Geom(M)}

Here we construct an automaton for a regular language in size-preserving bijection with $\Geom(M)$, from which we can compute the generating function for $\Geom(M)$. The map $\phi_M \colon \Sigma^* \to \Geom(M)$ is already surjective but has two sources of non-injectivity. One is that a given permutation in $\Geom(M)$ can have several different griddings, extending to distinct elements of $\Geom^\sharp(M)$. As in \cite{albert2013geometric}*{\S 8}, we will handle this by defining a linear order on the set of griddings of a given permutation, and only keep the word corresponding to the minimal gridding; the key difference is that we will exhibit an explicit sentence defining these minimal gridded permutations. The other source of non-injectivity is that, given two points in an $M$-gridded permutation in distinct rows and columns, exchanging their relative distance to the origins of their respective cells might not change the $M$-gridded permutation, but will correspond to different words; see \cite{albert2013geometric}*{Figure 7} for an example. This is already handled constructively through the theory of trace monoids in \cite{albert2013geometric}*{\S 7}.

\begin{definition}
	Let $\clex$ be the linear order on $\Sigma$ defined by $C_i \clex  C_j$ if $C_i$ is below $C_j$, or they are in the same row and $C_i$ is left of $C_j$.
	
	We now define, for each $\pi \in \Geom(M)$, a linear order $\lex$ on the set of elements of $\Geom^\sharp(M)$ isomorphic to $\pi$ as permutations. Let $\sigma^\sharp \lex \tau^\sharp$ if, letting $p$ be the $<_1$-least point such that $\sigma^\sharp$ and $\tau^\sharp$ place $p$ in different cells, we have $p \in C_i$ in $\sigma^\sharp$ and $p \in C_j$ in $\tau^\sharp$ with $C_i \clex C_j$.
\end{definition}

\begin{lemma} \label{lem:MinM}
	There is an $\LL_{P, M}$-sentence $\mathrm{Min}_M$, computable from $M$, such that for every $\pi^\sharp \in \Geom^\sharp(M)$, $\pi^\sharp \models \mathrm{Min}_M \iff \pi^\sharp$ is $\lex$-minimal.
\end{lemma}
\begin{proof}
	Let $\mathrm{Min}_M$ be the following sentence.
\begin{align*}
	\forall Y_1,& \dots, Y_n (\Mon_M(Y_1, \dots, Y_n) \wedge \mathrm{Acyc}_{D(Y_1, \dots, Y_n; x, y)} \ra \forall x (x \in Y_j \ra \\
	& ((\bigvee_{C_i \clexeq C_j} x \in C_i) \vee 
	\exists y (y <_1 x  \wedge \bigvee_{\ell \in [n]} (y \in Y_\ell \wedge \bigvee_{C_k \clex C_\ell} y \in C_k)))))															
\end{align*}
For its correctness, we remark that the sentence can be viewed quantifying over all possible griddings of its domain with each $Y_i$ corresponding to the cell $C_i$. That is, the proof of Proposition \ref{prop:msoDef} shows that any $Y_1, \dots, Y_n$ satisfying the $\Mon_M \wedge \mathrm{Acyc}$ clause of $\mathrm{Min}_M$ yields a gridding such that $x \in Y_i$ if and only if $x \in C_i$ for each $i \in [n]$, and conversely every gridding yields a choice of $Y_1, \dots, Y_n$ satisfying the first line such that $x \in Y_i$ if and only if $x \in C_i$.

Given this, $\pi^{\sharp} \models \mathrm{Min}_M$ states that for every alternative gridding of $\pi^{\sharp}$, every point either moves to a cell that is at least as large with respect to $\preceq_{lex}$ as its cell in $\pi^{\sharp}$, or has another point to its left that has moved to a cell that is strictly larger with respect to $\preceq_{lex}$ than that point's cell in $\pi^{\sharp}$.
	\end{proof}

\begin{lemma}
	There is an $\LL_\Sigma$-sentence $\mathrm{Bij}_M$, computable from $M$, defining a set of words in size-preserving bijection with $\Geom(M)$.
\end{lemma}
\begin{proof}
	By \cite[Corollary 7.3]{albert2013geometric}, there is a regular language $L^\sharp$ in size-preserving bijection with $\Geom^\sharp(M)$. We will show that we can compute an  $\LL_\Sigma$-sentence $\mathrm{Bij}_{M^\sharp}$ defining $L^\sharp$. By Theorem \ref{thm:interpRed} and Lemmas \ref{lem:interp} and \ref{lem:MinM}, there is an $\LL_\Sigma$-sentence $\mathrm{Min}'_M$ such that $w \models \mathrm{Min}'_M \iff \phi^\sharp(w) \models \mathrm{Min}_M$. Then we can take $\mathrm{Bij}_M = \mathrm{Bij}_{M^\sharp} \wedge \mathrm{Min}'_M$.
	
	The proof of \cite[Corollary 7.3]{albert2013geometric} relies on \cite[Proposition 1.2.2]{diekert1990combinatorics}, which yields the following description of $L^\sharp$. First, define a binary relation $I \subset \Sigma \times \Sigma$ that holds when the corresponding cells of the inputs are in different rows and columns. Then $w \in \Sigma^*$ is in $L^\sharp$ if for every factor $aub$ of $w$ with $a, b \in \Sigma, u \in \Sigma^*$ such that all $x \in au$ satisfy $I(x,b)$, we have $a \clex b$. This is easily definable by a sentence computable from $M$.
	\end{proof}

\begin{theorem} \label{thm:compgf}
	The generating function of $\Geom(M)$ is computable from $M$.
\end{theorem}
\begin{proof}
	Given $M$, we can compute $\mathrm{Bij}_M$, and thence a finite automaton accepting a language in size-preserving bijection with $\Geom(M)$ by Theorem \ref{thm:MSOreg}. The generating function for  the language accepted by this automaton can then be computed by the transfer-matrix method (e.g. \cite[Corollary 1.4.2]{ardila2015algebraic}).
	\end{proof}

\section{Some extensions}  \label{sec:ext}

We note that our results easily extend to certain subclasses of geometric grid classes, as well as to their substitution closures. We begin with a definition capturing a notion of indecomposability for permutations.

\begin{definition}
	 A permutation $\pi$ is {\em simple} if there does not exist $X \subsetneq \pi$ of size at least two such that for all $x, y \in X$ and $z \in \pi \bs X$, we have that $z <_i x \iff z <_i y$ for $i \in \set{1,2}$.
\end{definition}

Note that this definition can be easily translated into an MSO sentence $\mathrm{Simple}$. The idea of the notion is that if a permutation is not simple, then it can be obtained from a simpler quotient permutation by ``inflating'' the points into permutations. An inflation of $\pi$ can be viewed as blowing up the points of $\pi$ into boxes with non-overlapping axis-projections, and placing (possibly distinct) permutations in each box (see \cite[Figure 19]{vatter2015permutation} for an inflation of (2413)).

Our proofs in the previous sections show the following stronger result.

\begin{theorem} \label{thm:strongComp}
	Let $\theta_M$ be an $\LL_P$-sentence computable from $M$, and let $\CC = \set{\pi \in \Geom(M) | \pi \models \theta_M}$. Then the generating function of $\CC$ is rational and computable from $M$. If $\CC$ is downwards-closed under embeddings, then the basis of $\CC$ is computable from $M$.
	
	In particular, we may take $\CC$ to be the set of simple or (skew) sum-indecomposable elements of $\Geom(M)$, or we may take $\CC = \Geom(M) \cap \Av(\BB)$ for a given finite set $\BB$ of permutations.
\end{theorem}

We now turn to the {\em substitution closure} of a geometric grid class. Intuitively, this is the closure of $\Geom(M)$ under the ``inflation'' procedure described above. We refer to \cite{albert2015inflations}*{\S 2} for a formal definition, but we will only need the following results concerning the basis of the substitution closure of $\Geom(M)$.

\begin{theorem}[\cite{albert2015inflations}*{Proposition 2.7, proof of Proposition 2.9, Theorem 4.4}] \label{thm:scBasis}
	The substitution closure of $\Geom(M)$ has a finite basis. The basis permutations are the minimal simple permutations not contained in $\Geom(M)$. Furthermore, there is a matrix $A$, independent of $M$, such that every basis permutation either belongs to $\Geom(M)^{+1}$ or to $\Geom(A)$.
	\end{theorem}

\begin{theorem} \label{thm:compSC}
The basis of the substitution closure of $\Geom(M)$ is computable from $M$.
\end{theorem}
\begin{proof}
As in Theorem \ref{thm:compBasis}, we compute a matrix $N_0$ so that $\Geom(M)^{+1} \subset \Geom(N_0)$, and thus a matrix $N$ so that $\Geom(M)^{+1} \cup \Geom(A) \subset \Geom(N)$. We can then essentially repeat the proof of Theorem \ref{thm:compBasis}, but we replace $\mathrm{Basis}_{\BB, M}$ by the sentence 
\[(\neg \Geom_M \wedge \mathrm{Simple}) \ra \bigvee_{B \in \BB}(\text{there exist elements inducing a copy of $B$})\] 
and as we iterate through all permutations, we consider only simple ones.
\end{proof}

We may also give an upper bound on the size of the basis elements of the substitution closure of $\Geom(M)$ as in Proposition \ref{prop:bound}.

\section{Quantifier-free interpretations}
The primary fact about geometric grid classes leading to their good properties is that they are interpretable in the class of $\Sigma$-words by the particularly simple (in particular, quantifier-free) formulas defined in Lemma \ref{lem:interp}. In this section, we show that the existence of such a quantifier-free interpretation actually characterizes subclasses of geometric grid classes. We will find it easiest to prove the analogous result for graph classes of bounded lettericity and then use their connection to geometric grid classes.  Like the main result of \cite{alecu2022letter} recalled below, this also shows that graph classes of bounded lettericity and subclasses of geometric grid classes of permutations are in a sense the same classes in their respective categories.

\begin{definition}
	Given a finite alphabet $\Sigma$, we call $\PP \subset \Sigma^2$ a {\em decoder}. Given $w \in \Sigma^*$ and a decoder $\PP$, the {\em letter graph of $w$} is the graph with the same domain as $w$ where the edge relation is the symmetric closure of the following relation defined using the language of $\Sigma$-words from Section \ref{sec:preMSO}: put $x$ adjacent to $y$ if $w \models x < y \wedge U_i(x) \wedge U_j(y)$ and $(U_i, U_j) \in \PP$.
	
	A graph $G$ has {\em lettericity} at most $k$ if it is the letter graph of some $w \in \Sigma$ with respect to some decoder, where $|\Sigma| \leq k$.
\end{definition}

\begin{theorem}[\cite{alecu2020letter}*{Theorem 3}, \cite{alecu2022letter}*{Theorem 1.1}] \label{thm:lettgrid}
	A permutation class $\CC$ is a subset of a geometric grid class if and only if the corresponding class of inversion graphs $\GG_\CC$ has bounded lettericity.
\end{theorem} 

\begin{definition}
Recalling the definition of an interpretation $I \colon \DD \to \CC$ of $\CC$ in $\DD$ from Section \ref{sec:preMSO}, we will say $I$ is a  {\em quantifier-free interpretation} if each $\LL_\DD$-formula $R_S$ is quantifier-free.
\end{definition}

\begin{proposition} \label{prop:qfinterp}
	Let $\CC$ be a graph class. Then $\CC$ has lettericity at most $k$ if and only if there is a quantifier-free interpretation $I \colon \Sigma^* \to \CC$, for some $\Sigma$ with $|\Sigma| = k$.
	
	Let $\DD$ be a permutation class. Then $\DD \subset \Geom(M)$ for some $M$ if and only if there is a quantifier-free interpretation $I \colon \Sigma^* \to \DD$, for some $\Sigma$ with $|\Sigma|$ finite.
\end{proposition}
\begin{proof}
	Let $\CC$ be a graph class of lettericity at most $k$. Then the decoder for $\CC$ immediately provides the desired quantifier-free interpretation.
	
	Now let $\CC$ be a graph class for which there exists such an interpretation $I$ and quantifier-free formula $R_E(x,y)$. Writing $R_E(x,y)$ in disjunctive normal form, we have $R_E(x,y) := \bigvee_i \phi_i(x,y)$, where each $\phi_i$ is a conjunction of relations from the language of $\Sigma$-words and their negations. Each such $\phi_i$ can specify a subset of allowable unary relations for each of $x$ and $y$, and can specify either that $x < y$, that $y < x$, or neither. From this, the desired decoder is easily constructed.

	If $\DD \subset \Geom(M)$, then Lemma \ref{lem:interp} shows that $\phi_M$ provides the desired quantifier-free interpretation.
	
	Let $\DD$ be a permutation class for which there exists a quantifier-free interpretation $I \colon \Sigma^* \to \DD$. Let $\GG_\DD$ be the class of inversion graphs of $\DD$ and let $I' : \DD \to \GG_\DD$ be the interpretation with $R_E(x,y) := (x <_1 y \wedge x >_2 y) \vee (x <_2 y \wedge x >_1 y)$. Then $I' \circ I : \Sigma^* \to \GG_\DD$ is a quantifier-free interpretation, and so $\GG_\DD$ is a graph class of lettericity at most $k$. By Theorem \ref{thm:lettgrid}, $\DD$ is a subset of a geometric grid class, as desired.
\end{proof}

\section{Questions}

We close with some questions, the first suggested in the conclusion of \cite{albert2013geometric}.

\begin{question}
	Given a finite set $\BB$ of permutations, is it decidable whether $\Av(\BB)$ is contained in a geometric grid class?
\end{question}

It would be sufficient to show that if $\Av(\BB)$ is contained in a geometric grid class, then it is contained in a class $\Geom(M)$ where $|M|$ is bounded in terms of the size of the largest element of $\BB$.

\begin{question}
Can the upper bound on the size of a basis element in Proposition \ref{prop:bound} be improved to a polynomial in $|M|$? Or even to a linear bound?
	\end{question}

Since this paper first appeared, \cite{alecu2024lettericity} has carried out a parallel analysis for the class of graphs of lettericity at most $k$. Furthermore, \cite{alecu2024lettericity} gives a $2^{poly(k)}$ upper bound on the size of the basis elements for these classes by direct combinatorial arguments.

Viewing words as colored linear orders, the interpretations in Proposition \ref{prop:qfinterp} may be decomposed into a non-deterministic coloring (of a linear order) followed by an interpretation, which together is called a {\em transduction}. We refer to \cite{braunfeld2022first} for more details. This suggests two orthogonal ways of generalizing geometric grid classes: by starting with a more general class such as trees, or by allowing interpretations beyond quantifier-free ones. This leads to permutation classes of bounded (linear) clique-width studied in \cite{opler2022structural}. We ask whether the substitution closure analyzed in Section \ref{sec:ext} fits into this picture.

\begin{question}
	Is there a characterization of the substitution closure of either geometric grid classes or graph classes of bounded lettericity in terms of (possibly restricted) first-order interpretations or transductions?
	\end{question}

 \bibliographystyle{alpha}
\bibliography{../Bib}      
\end{document}